\documentclass[preprint]{elsarticle} 

\usepackage{amssymb}
\usepackage{amsmath} 
\usepackage{eucal}
\usepackage{amsfonts}
\usepackage{amsthm}
\usepackage[colorlinks]{hyperref} 

\newtheorem{thm}{Theorem}[section]
\newtheorem{cor}[thm]{Corollary}
\newtheorem{lem}[thm]{Lemma}
\newtheorem{prop}[thm]{Proposition}

\theoremstyle{definition}

\newtheorem{ex}[thm]{Example}
\newtheorem{remk}[thm]{Remark}

\newcommand{\HH}{\mathcal{H}}
\newcommand{\KK}{\mathcal{K}}

\newcommand{\R}{\mathbb{R}}
\newcommand{\C}{\mathbb{C}}

\newcommand{\leqs}{\leqslant}
\newcommand{\geqs}{\geqslant}

\newcommand{\ap}{\alpha}
\newcommand{\bt}{\beta}

\newcommand{\ep}{\epsilon}
\newcommand{\ld }{\lambda}

\newcommand{\sm}{\,\sigma\,}

\newcommand{\hm}{\, ! \,}

\newcommand{\eop}{\;\hfill $\square$}

\begin{document}


\begin{frontmatter}

\title{Characterizations of Connections for Positive Operators}

\author[CU]{Pattrawut Chansangiam} 
\ead{kcpattra@kmitl.ac.th}
\author[CU]{Wicharn Lewkeeratiyutkul \corref{cor1}}
\ead{Wicharn.L@chula.ac.th}

\address[CU]{Department of Mathematics and Computer Science, Faculty of Science,
Chulalongkorn University, Bangkok 10330, Thailand}

\cortext[cor1]{correponding author, Tel. +66 2 2185161}

\begin{abstract}

An axiomatic theory of operator connections and operator means was investigated by Kubo and Ando in 1980.
A connection is a binary operation for positive operators satisfying
the monotonicity, the transformer inequality and the joint-continuity from above.
In this paper, we show that the joint-continuity assumption
can be relaxed to some conditions which are weaker than the separate-continuity.
This provides an easier way for checking whether a given binary opertion is a connection.
Various axiomatic characterizations of connections are obtained.
We show that the concavity is an important property of a connection
by showing that the monotonicity can be replaced by the concavity
or the midpoint concavity. Each operator connection induces a unique scalar connection.
Moreover, there is an affine order isomorphism between connections and induced connections.
This gives a natural viewpoint to define any named means.

\end{abstract}


\end{frontmatter}


\section{Introduction}

Throughout, let $\HH$ denote an infinite-dimensional Hilbert space.
Let $B(\HH)$ be the algebra of bounded linear operators on $\HH$
and $B(\HH)^+$ its positive cone. Equip $B(\HH)$ with the usual positive semidefinite ordering.
Unless otherwise stated, any limit in  $B(\HH)$ is taken in the strong-operator topology.

The concept of means is one of the most familiar concepts in mathematics.
It is proved to be a powerful tool from theoretical as well as practical points of view.
The theory of scalar means was developed since the ancient Greek by the Pythagoreans 
(via the method of proportions, see \cite{Toader-Toader})
until the last century by many famous mathematicians.
The theory of connections and means for matrices and operators started when the concept of
parallel sum was introduced in \cite{Anderson-Duffin} for analyzing electrical networks.
The \emph{parallel sum} of two positive definite matrices (or invertible positive operators) $A$ and $B$ is
defined by
\begin{align*}
    (A,B) \mapsto (A^{-1}+B^{-1})^{-1}.
\end{align*}
For general positive operators, use continuity:
\begin{align*}
    (A,B) \mapsto \lim_{\ep \downarrow 0}(A_{\ep}^{-1}+B_{\ep}^{-1})^{-1},
\quad A_{\ep} \equiv A+\ep I, B_{\ep} \equiv B+\ep I.
\end{align*}
The \emph{harmonic mean}, denoted by $!$, for positive operators is the twice parallel sum.
The \emph{geometric mean} of two positive semidefinite matrices (or positive operators)
is defined and studied in \cite{Ando}:
\begin{align*}
    A \# B = \lim_{\ep \downarrow 0} {A_{\ep}}^{1/2} ({A_{\ep}}^{-1/2}B_{\ep} {A_{\ep}}^{-1/2})^{1/2} {A_{\ep}}^{1/2}
\end{align*}
where $A_{\ep} \equiv A+\ep I, B_{\ep} \equiv B+\ep I$.
In \cite{Ando}, geometric means and harmonic means played crucial roles in the study of concavity and monotonicity 
of many interesting maps between matrix spaces.
Another important mean in mathematics, namely the power mean, was considered in \cite{Bhagwat-Subramanian}.


A study of operator means in an abstract way
was given by Kubo and Ando \cite{Kubo-Ando}. Let $\KK$ be a Hilbert space, here we do not assume that
$\dim \KK = \infty$.
A \emph{connection} is a binary operation $\sm$ on $B(\KK)^+$
such that for all positive operators $A,B,C,D$:
\begin{enumerate}
	\item[(M1)] \emph{monotonicity}: $A \leqs C, B \leqs D \implies A \sm B \leqs C \sm D$
	\item[(M2)] \emph{transformer inequality}: $C(A \sm B)C \leqs (CAC) \sm (CBC)$
	\item[(M3)] \emph{joint--continuity from above}:  for $A_n,B_n \in B(\KK)^+$,
                if $A_n \downarrow A$ and $B_n \downarrow B$,
                 then $A_n \sm B_n \downarrow A \sm B$.
\end{enumerate}
Typical examples of connection are the sum $(A,B) \mapsto A+B$ and the parallel sum.
A \emph{mean} is a connection $\sigma$ such that $I \sm I = I$.
The followings are examples of means in practical usage:  
	\begin{itemize}
		\item	$t$-weighted arithmetic means: $A \triangledown_{t} B = (1-t)A + tB$
		\item	$t$-weighted geometric means: 
    			$A \#_{t} B =  A^{1/2} 
    			({A}^{-1/2} B  {A}^{-1/2})^{t} {A}^{1/2}$
		\item	$t$-weighted harmonic means: $A \,!_t\, B = [(1-t)A^{-1} + tB^{-1}]^{-1}$ 
		\item	logarithmic mean: $(A,B) \mapsto A^{1/2} f (A^{-1/2} B A^{-1/2}) A^{1/2}$ where $f(x) = (x-1)/\log x$.
	\end{itemize}

The theory of operator monotone functions plays a crucial role in Kubo-Ando theory of connections and means.
A continuous real-valued function $f$ on an interval $I$ is called an \emph{operator monotone function} if
one of the following equivalent conditions holds:
\begin{enumerate}
    \item[(i)]   $A \leqs B \implies f(A) \leqs f(B)$ for all Hermitian matrices $A,B$ of all orders whose spectrums are contained in $I$;
    \item[(ii)]   $A \leqs B \implies f(A) \leqs f(B)$ for all Hermitian operators $A,B \in B(\HH)$ whose spectrums are contained in $I$ and for some infinite-dimensional Hilbert space $\HH$;
    \item[(iii)]   $A \leqs B \implies f(A) \leqs f(B)$ for all Hermitian operators $A,B \in B(\KK)$ whose spectrums are contained in $I$ and for all Hilbert spaces $\KK$.
\end{enumerate}
This concept was introduced in \cite{Lowner}; see also \cite{Bhatia,Hiai,Hiai-Yanagi}.
Denote by $OM(\R^+)$ the set of operator monotone functions from $\R^+=[0,\infty)$ to itself.

In \cite{Kubo-Ando}, a connection $\sigma$ on $B(\HH)^+$ can be characterized as follows:
\begin{itemize}
	\item	There is an $f \in OM(\R^+)$ satisfying
				\begin{align}
    			f(x)I = I \sm (xI), \quad x \in \R^+. \label{eq: f(t)I = I sm tI}
				\end{align}
	\item	There is an $f \in OM(\R^+)$ such that
				\begin{align}
						A \sm B = A^{1/2} f(A^{-1/2} B A^{-1/2}) A^{1/2}, \quad A,B >0. \label{eq: A sm B in term of f}
				\end{align}
	\item	There is a finite Borel measure $\mu$ on $[0,\infty]$ such that
				\begin{align}
    			A \sm B = \ap A + \bt B + \int_{(0,\infty)} \frac{\ld+1}{2\ld} \{ (\ld A) \,!\,B\}\, d\mu(\ld)
    			\label{eq: formula of connection}
				\end{align}
				where the integral is taken in the sense of Bochner, $\ap = \mu(\{0\})$ and $\bt = \mu(\{\infty\})$.
\end{itemize}
In fact, the functions $f$ in \eqref{eq: f(t)I = I sm tI} and \eqref{eq: A sm B in term of f} are unique and the same.
We call $f$ the \emph{representing function} of $\sigma$.
From the integral representation \eqref{eq: formula of connection}, every connection $\sigma$ is concave in the sense that
\begin{align}
	(tA+(1-t)B) \sm (tA'+(1-t)B') \geqs t(A \sm A')+(1-t)(B \sm B') \label{eq: concavity of sm}
\end{align}
for all $A,B \geqs 0$ and $t \in (0,1)$.
Moreover, the map $\sm \mapsto f$ is an affine order-isomorphism between the set of connections and $OM(\R^+)$.
Here, the order-isomorphism means that when $\sm_i \mapsto f_i$ for $i=1,2$,
			$A \,\sigma_1\, B \leqs A \,\sigma_2\, B$ for all $A,B \in B(\HH)^+$ if and only if $f_1 (x) \leqs f_2 (x)$ for all $x \in \R^+$.

The mean theoretic approach has various applications.
The concept of means can be used to obtain the monotonicity, concavity and convexity of interesting maps 
between matrix algebras or operator algebras (see the original idea in \cite{Ando}).
The fact that the map $f \mapsto \sigma$ is an order isomorphism can be used to obtain simple proofs of 
operator inequalities concerning means. For example, the arithmetic--geometric--logarithmic--harmonic means
inequalities are obtained from applying this order isomorphism to the scalar inequalities
	\begin{align*}
		\frac{2x}{1+x} \leqs x^{1/2} \leqs \frac{x-1}{\log x} \leqs \frac{1+x}{2}, \quad x > 0, x \neq 1.
	\end{align*}
The concavity of general connections serves simple proofs of operator versions of H\"{o}lder inequality, Cauchy-Schwarz inequality, 
Minkowski's inequality, Aczel's inequality, Popoviciu's inequality and Bellman's inequality (e.g. \cite{Mond}).
The famous Furuta's inequality and its generalizations are obtained from axiomatic properties of connections 
(e.g. \cite{Fujii_Furuta ineq,Fujii-Kamei_Furuta ineq,Ito-Kamei_Furuta ineq}).  
Kubo-Ando theory can be applied to matrix and operator equations since harmonic and geometric means can be viewed as solutions of certain operator equations. See some examples of applications in
\cite{Anderson-Morley_matrix equation, Lim_matrix eq}.
It also plays an important role in noncommutative information theory.
A relative operator entropy was defined in \cite{Fujii_entropy} to be the connection corresponding 
to the operator monotone function $x \mapsto \log x$.
See more information in \cite[Chapter IV]{Bhatia} and therein references.

Kubo-Ando definition of a connection is a binary operation satisfying axioms (M1), (M2) and (M3).
In this work, we show that some of the axioms can be weakened.
Moreover, we provide alternative sets of axioms involving concavity property.
This gives a direct tool for studying operator inequalities.

Consider the following axioms:
\begin{itemize}
    \item[(M3$'$)] for each $A,X \in B(\HH)^+$,  if $A_n \downarrow A$, then
                  $A_n \sm X \downarrow A \sm X$ and $I \sm A_n \downarrow I \sm A$;
    \item[(M3$''$)] for each $A,X \in B(\HH)^+$,  if $A_n \downarrow A$, then
                  $X \sm A_n \downarrow X \sm A$ and $A_n \sm I \downarrow A \sm I$;
    \item[(M4)]   \emph{concavity}: $(tA+(1-t)B) \sm (tA'+(1-t)B') \geqs t(A \sm A')+(1-t)(B \sm B')$
                    for $t \in (0,1)$;
    \item[(M4$'$)]   \emph{midpoint concavity}: $(A+B)/2 \,\sm\, (A'+B')/2 
    					\geqs [(A \sm A')+ (B \sm B')]/2$.
\end{itemize}

Note that condition (M3$'$) is one of the axiomatic properties of abstract solidarity introduced in \cite{Fujii_solidarities}.
The definition of a connection can also be relaxed as follows:

\begin{thm}\label{thm: Improve connection}
    Let $\sigma$ be a binary operation on $B(\HH)^+$. Then the following statements are equivalent:
    \begin{enumerate}
        \item[(i)]   $\sigma$ is a connection;
        \item[(ii)]    $\sigma$ satisfies (M1), (M2) and (M3$'$);
        \item[(iii)]    $\sigma$ satisfies (M1), (M2) and (M3$''$).
    \end{enumerate}
\end{thm}

Condition (M3$'$) or (M3$''$) is clearly weaker, and easier to verify, than the joint-continuity assumption (M3)
in Kubo-Ando definition. 

A connection can be axiomatically defined as follows.
Fix the transformer inequality (M2). 
We can freely replace the monotonicity (M1) by the concavity (M4) 
or the mid-point concavity (M4$'$). 
At the same time, we can use (M3$'$) or (M3$''$) instead of the joint-continuity (M3).

\begin{thm}\label{thm: Characterize connection}
    Let $\sigma$ be a binary operation on $B(\HH)^+$ satisfying (M2). Then the following statements are equivalent:
\begin{enumerate}
    \item[(1)]    $\sigma$ is a connection;
    \item[(2)]    $\sigma$  satisfies (M4) and (M3);
    \item[(3)]    $\sigma$  satisfies (M4) and (M3$'$);
    \item[(4)]    $\sigma$  satisfies (M4) and (M3$''$);
    \item[(5)]    $\sigma$  satisfies (M4$'$) and (M3);
    \item[(6)]    $\sigma$  satisfies (M4$'$) and (M3$'$);
    \item[(7)]    $\sigma$  satisfies (M4$'$) and (M3$''$).
\end{enumerate}
\end{thm}
This theorem gives different viewpoints of Kubo-Ando connections. 
It shows the importance of the concavity property of a connection.
Moreover, it asserts that the concepts of monotonicity and concavity are equivalent under suitable conditions.

Theorem \ref{thm: Improve connection} and Theorem \ref{thm: Characterize connection} 
are established in Sections 2 and 3, respectively.
Each operator connection induces a unique scalar connection on $\R^+$.
Furthermore, there is an affine order isomorphism between connections and induced connections.
This gives a natural way to define any named mean.
For example, a geometric mean on $B(\HH)^+$ is the mean
on $B(\HH)^+$ that corresponds to the usual geometric mean on $\R^+$.
The correspondence between connections and induced connections will be discussed in details in Section 4.


\section{Relaxing the definition of connection}

In this section, we show that the joint-continuity assumption in the definition of connection can
be relaxedly defined by (M3$'$) or (M3$''$), which are weaker than (M3).
Let $\sigma : B(\HH)^+ \times B(\HH)^+ \to B(\HH)^+$ be a binary operation.

\begin{remk} \label{rem: M2 imply Cong inv and Posit hom}
The transformer inequality (M2) implies
\begin{itemize}
    \item   \emph{congruence invariance}: $C(A \sigma B)C = (CAC) \sm (CBC)$ for $A,B \geqs 0$ and $C>0$;
    \item   \emph{positive homogeneity}: $\ap (A \sm B) = (\ap A) \sm (\ap B)$
    				for $A,B \geqs 0$ and $\ap \in (0,\infty)$.
\end{itemize}
\end{remk}

We say that $\sigma$ satisfies property (P) if
                    \begin{align*}
                        P(A \sm B) = (PA) \sm (PB) = (A \sm B)P
                    \end{align*} 
for any projection $P \in B(\HH)^+$ commuting with $A,B \in B(\HH)^+$.
A function $f: \R^+ \to \R^+$ and $\sigma$ are said to satisfy property (F) if for any $x \in \R^+$,
\begin{align*}
	f(x) I  =  I  \sm (x I).  
\end{align*}


\begin{lem} \label{lem: M3'+F implies f is cont}
    Let $f: \R^+ \to \R^+$ be an increasing function. If $\sigma$ satisfies the positive homogeneity, (M3$'$) and (F),
    then $f$ is continuous.
\end{lem}
\begin{proof}
    To show that $f$ is right continuous at each $x\in\R^+$, consider a sequence $x_n$ in $\R^+$ such that $x_n \downarrow x$. Then by (M3$'$)
\begin{align*}
    f(x_n)  I  =  I  \sm (x_n I)  \downarrow  I  \sm (x I)  = f(x) I ,
\end{align*}
i.e. $f(x_n) \downarrow f(x)$. To show that $f$ is left continuous at each $x>0$, consider a sequence $x_n >0$ such that $x_n$ is increasing and $x_n \to x$. Then $x_n^{-1} \downarrow x^{-1}$ and
\begin{align*}
    \lim x_n^{-1} f(x_n)  I  &= \lim \; x_n^{-1}( I  \sm x_n I )
        = \lim \; (x_n^{-1} I ) \sm   I
        = (x^{-1} I ) \sm   I \\
        &= x^{-1}(I \sm xI)
        = x^{-1} f(x)I
\end{align*}
That is $x \mapsto x^{-1}f(x)$ is left continuous
and so is $f$.
\end{proof}

\begin{lem} \label{thm: f(A) = 1 sm a}
    Let $\sigma$ be a binary operation on $B(\HH)^+$ satisfying (M3$'$) and (P).
    If $f: \R^+ \to \R^+$ is an increasing continuous function where $\sigma$ and $f$ satisfy (F),
    then $f(A) =  I  \sm A$ for any $A \in B(\HH)^+$.
\end{lem}
\begin{proof}
    First consider $A \in B(\HH)^+$ in the form $\sum_{i=1}^m \ld_i P_i$ where $\{P_i\}_{i=1}^m$ is an orthogonal
family of projections with sum $ I $ and $\ld_i>0$ for all $i=1,\dots,m$. Since each $P_i$ commutes with $A$,
we have by the property (P) that
\begin{align*}
     I  \sm A &= \sum {P_i} ( I  \sm A) = \sum P_i \sm P_i A = \sum P_i \sm \ld_i P_i \\
            &= \sum P_i ( I  \sm \ld_i  I ) = \sum f(\ld_i)P_i = f(A).
\end{align*}
Now, consider $A \in B(\HH)^+$. Then there is a sequence $A_n$ of strictly positive operators in the above form
such that $A_n \downarrow A$. Then $ I  \sm A_n \downarrow  I  \sm A$ and $f(A_n)$ converges strongly to $f(A)$. 
Hence, $ I  \sm A = \lim  I  \sm A_n = \lim f(A_n) = f(A)$.
\end{proof}

Denote by $BO(M1, M2, M3')$ the set of binary operations satisfying axioms (M1), (M2) and (M3$'$).
Similar notations are applied for other axioms. 

\begin{proof}[Proof of Theorem \ref{thm: Improve connection}:]
We have known that (i) $\Rightarrow$ (ii) and (i) $\Rightarrow$ (iii).

(ii) $\Rightarrow$ (i).
    Let $\sigma \in BO(M1, M2, M3')$. As in \cite{Kubo-Ando},
    the conditions (M1) and (M2) imply that $\sigma$ satisfies (P) and there is a function $f:\R^+ \to \R^+$
    subject to (F).
    If $0 \leqs x_1 \leqs x_2$, then by (M1)
    \begin{align*}
        f(x_1)I = I \sm (x_1 I) \leqs I \sm (x_2 I) = f(x_2)I,
    \end{align*}
    i.e. $f(x_1) \leqs f(x_2)$.
    Then the assumption (M3$'$) is sufficient to guarantee that $f$ is continuous by
    Lemma \ref{lem: M3'+F implies f is cont}.
Lemma \ref{thm: f(A) = 1 sm a} results in $f(A)=I \sigma A$ for all $A \geqs 0$. 
Now, (M1) and the fact that
$\dim \HH= \infty$ yield that $f$ is operator monotone.
The uniqueness of $f$ is obvious.
Thus, we establish a well-defined map
$\sigma \in BO(M1, M2, M3') \mapsto f \in OM(\R^+)$ such that  $\sigma$ and $f$ satisfy (F).

Now, given $f \in OM(\R^+)$, we construct $\sigma$ as in \cite{Kubo-Ando}:
\begin{align}
    A \sm B = \ap A + \bt B+\int_{(0,\infty)} \frac{\ld+1}{2\ld} \{ (\ld A) \hm B\} d\mu(\ld) \label{eq: a sm B = int}
\end{align}
where $\mu$ is the corresponding measure of $f$, $\ap=\mu(\{0\})$ and $\bt=\mu(\{\infty\})$.
Then $\sigma$ satisfies (M1), (M2), (M3$'$) and (F).
This shows that the map $\sigma \mapsto f$ is surjective.

To show the injectivity of this map, let $\sigma_1,\sigma_2 \in BO(M1,M2,M3')$ be such that $\sigma_i \mapsto f$
where, for each $t \geqs 0$,
\begin{align*}
    I \,\sigma_i\, (xI) = f(x)I, \quad i=1,2.
\end{align*}
Since $\sigma_i$ satisfies the property (P), we have $I \,\sigma_i\, A = f(A)$ for $A \geqs 0$
by Lemma \ref{thm: f(A) = 1 sm a}.
Since $\sigma_i$ satisfies the congruence invariance, we have that for $A>0$ and $B \geqs 0$,
\begin{align*}
    A \,\sigma_i\, B = A^{1/2}(I \,\sigma_i\, A^{-1/2}BA^{-1/2})A^{1/2} = A^{1/2} f(A^{-1/2}BA^{-1/2})A^{1/2}.
\end{align*}
By the limiting argument, we see that (M3$'$) implies $\sigma_1 = \sigma_2$.

Thus there is a bijection between $OM(\R^+)$ and $BO(M1, M2, M3')$. Every element in $BO(M1, M2, M3')$ has an integral representation \eqref{eq: a sm B = int}. Since the harmonic mean possesses (M3), so is any element in
$BO(M1, M2, M3')$.

(iii) $\Rightarrow$ (i). We can develop the analogous results when (M3$'$) is replaced by (M3$''$) by
swapping ``left'' and ``right.'' 
Indeed, we establish a one-to-one correspondence between $\sigma \in BO(M1, M2, M3'')$ and $g \in OM(\R^+)$, where 
\begin{align}
    g(x)I = (xI) \sm I, \quad x \in \R^+. \label{eq: g(x)I = I sm tI}
\end{align}
Here, \eqref{eq: g(x)I = I sm tI} plays the same role as property (F) in the proof of (ii) $\Rightarrow$ (i).
\end{proof}

\begin{remk}
		The representing function of a connection $\sigma$ in Kubo-Ando theory
    can be shown to be the function $f \in OM(\R^+)$ satisfying one of the following equivalent conditions for each $x \in \R^+$:
    \begin{enumerate}
        \item[(i)]   $f(x)I=I \sm (xI)$;
        \item[(ii)]   $f(x)P = P \sm (xP)$ for all projections $P$;
        \item[(iii)]   $f(x)A = A \sm (xA)$ for all $A>0$;
        \item[(iv)]   $f(x)A = A \sm (xA)$ for all $A \geqs 0$.
    \end{enumerate}
    There is also a one-to-one correspondence between connections $\sigma$ and operator monotone functions $g$ on $\R^+$ satisfying 
    \eqref{eq: g(x)I = I sm tI}. 
    Note that $g$ is the representing function of the transpose of $\sigma$.
    Indeed, the relationship between the representing function $f$ and the function $g$ 
    in \eqref{eq: g(x)I = I sm tI} is given by 
    \begin{align*}
        g(x) = xf(1/x). 
    \end{align*}
\end{remk}

\section{Characterizations of connections}

In this section, we give various characterizations of connections.
In order to prove Theorem \ref{thm: Characterize connection}, we need the following lemmas.

\begin{lem}  \label{prop: connection_basic prop}
If $\sigma \in BO(M2,M4')$, then for each $A,B,C,D \geqs 0$,
\begin{enumerate}
    \item[(i)]   $(A \sm B) + (C \sm D) \leqs (A + C) \sm  (B +D)$;
    \item[(ii)]   $A \leqs B$ implies $A \sm  I  \leqs B \sm  I  $ and
                        $ I \sm A \leqs  I \sm B$.
\end{enumerate}
\end{lem}
\begin{proof}
As in Remark \ref{rem: M2 imply Cong inv and Posit hom}, (M2) implies the positive homogeneity.
The fact (i) follows from the midpoint concavity (M4$'$) and positive homogeneity.
If $A \leqs B$, then by (i),
\begin{align*}
     I  \sm B = ( I +0) \sm (A+B-A) \geqs ( I  \sm A) + (0 \sm (B-A)) \geqs  I  \sm A.
\end{align*}
Similarly, $B \sm  I  \geqs A \sm  I $.
\end{proof}

\begin{lem}\label{lem: p commutes with A,B}
    If $\sigma \in BO(M2,M4')$, then $\sigma$ satisfies (P).
\end{lem}
\begin{proof}
    Let $P$ be a projection such that $AP=PA$ and $BP=PB$. We have $A = PAP+( I -P)A( I -P)$ and
    $B = PBP+( I -P)B( I -P)$. Then by Lemma \ref{prop: connection_basic prop} (i) and (M2)
\begin{align}
    A \sm B  &\geqs (PAP \sm PBP) + (( I -P) A ( I -P) \sm ( I -P) B ( I -P)) \label{eq: lem proj 1}\\
        &\geqs P(A \sm B)P+( I -P)(A \sm B)( I -P). \label{eq: lem proj 2}
\end{align}
Consider $C = A \sm B - P(A \sm B)P - ( I -P)(A \sm B)( I -P)$.
Then $C$ is positive and $PCP=0=( I -P)C( I -P)$, which implies $C^{1/2}P=0=C^{1/2}(I-P)$.
Hence, $CP=0=C(I -P)$ and $C=0$, meaning that
\begin{align*}
    A \sm B = P(A \sm B)P + ( I -P)(A \sm B)( I -P).
\end{align*}
It follows that $P(A \sm B) = P(A \sm B)P = (A \sm B)P$.
Furthermore, inequalities \eqref{eq: lem proj 1} and \eqref{eq: lem proj 2} become equalities, 
which implies $P(A \sm B)P = (PAP) \sm (PBP) = (PA) \sm (PB)$.
\end{proof}

\begin{lem} \label{lem: s sm t}
    If $\sigma \in BO(M2, M4')$, then there exists a unique
    binary operation $\tilde{\sigma}$ on $\R^+$ subject to the same properties and
    \begin{equation}
        (x  I ) \sm (y  I ) = (x \,\tilde{\sigma}\, y) I, \quad x,y \in \R^+. \label{eq: sI sm tI =}
    \end{equation}
\end{lem}
\begin{proof}
    Note that any projection on $\HH$ commutes with $x  I $ and $y  I $ for any $x,y \in \R^+$.
By Lemma \ref{lem: p commutes with A,B}, $(xI)  \sm (yI) $ commutes with any projection in $B(\HH)$.
The spectral theorem implies that $(xI)  \sm (yI) $ is a nonnegative multiple of identity, i.e.
there exists a $k \in \R^+$ such that $(x  I ) \sm (y  I ) = k I $.
If there is a $k' \in \R^+$ such that $(x  I ) \sm (y  I ) = k' I $, then $k'=k$.
Hence, each connection $\sigma$ on $B(\HH)^+$ induces a unique binary operation 
$\tilde{\sigma}: \R^+ \times \R^+ \to \R^+$ satisfying \eqref{eq: sI sm tI =}.
It is routine to check that $\tilde{\sigma}$ satisfies (M2) and (M4$'$).
\end{proof}

\begin{prop}  \label{thm: connection gives operator mon}
	If $\sigma \in BO(M2, M3',M4')$,
then there exists a unique $f \in OM(\R^+)$ satisfying (F).
In fact, $f(x) = 1 \tilde{\sm} x$ for $x \in \R^+$.
\end{prop}
\begin{proof}
    Define $f: \R^+ \to \R^+$ by $x \mapsto 1 \tilde{\sm} x$. The function $f$ is well-defined, unique and
satisfying  (F) by Lemma \ref{lem: s sm t}. If $0 \leqs x_1 \leqs x_2$,
then Lemma \ref{prop: connection_basic prop} (ii) implies
\begin{align*}
    f(x_1)I = I \sm (x_1 I) \leqs I \sm (x_2 I) = f(x_2)I,
\end{align*}
i.e. $f(x_1) \leqs f(x_2)$.
The continuity of $f$ is assured by Lemma \ref{lem: M3'+F implies f is cont}.
Then Lemma \ref{thm: f(A) = 1 sm a} implies $f(A) = I \sigma A$ for all $A \geqs 0$.
If $A,B \in B(\HH)^+$ are such that $A \leqs B$, then $f(A) = I \sm A \leqs  I  \sm B = f(B)$, again by
Lemma \ref{prop: connection_basic prop} (ii).
Since $\HH$ is infinite dimensional, $f$ is operator monotone.
\end{proof}

\noindent \emph{Proof of Theorem \ref{thm: Characterize connection}:}
    We have known that (1) implies (2)-(7). 
    It suffices to show that (6) implies (1).
    Assume that $\sigma \in BO(M2, M3',M4')$.
    Our aim is to construct a bijection
    between $BO(M2, M3',M4')$ and $OM(\R^+)$.
    Proposition \ref{thm: connection gives operator mon} assures that the map
    $\sigma \in BO(M2, M3',M4') \mapsto f \in OM(\R^+)$ where $\sigma$ and $f$ satisfy the property (F) is well-defined.
    This map is surjective via the same method as the construction in Theorem \ref{thm: Improve connection}.
    The injectivity of this map can be proved by using the same argument as in the proof of 
    Theorem \ref{thm: Improve connection}. 
    Here, the property (P) of
    $\sigma \in BO(M2, M3',M4')$ is fulfilled by
    Lemma \ref{lem: p commutes with A,B}. 
    Hence, we are allowed to consider only the binary operations constructed from operator monotone functions on $\R^+$. 
    Thus, $\sigma$ takes the form \eqref{eq: a sm B = int}. By passing the properties (M1) and (M3) of the harmonic mean through the integral representation, $\sigma$ also satisfies those properties.
\eop

\section{Induced connections} 

In this section, we consider the relationship between connections and their induced connections.

Each connection $\sigma$ on $B(\HH)^+$ induces a unique connection $\tilde{\sigma}$
on $\R^+=B(\C)^+$ satisfying
\begin{align*}
    (x \tilde{\sm} y)I = (xI) \sm (yI), \quad x,y \in \R^+.
\end{align*}
We call $\tilde{\sm}$ the \emph{induced connection} of $\sigma$.
Using the positive homogeneity of $\sigma$ and Lemma \ref{thm: f(A) = 1 sm a}, we have
\begin{align}
	x \tilde{\sm} y = xf(y/x) = xf(y/x), \quad x,y>0. \label{prop: induced con_prop}
\end{align}

\begin{prop} \label{cor: operator mon and formula of induc con}
    Each connection $\sm$ on $B(\HH)^+$  gives rise to an operator monotone function
    $x \mapsto 1 \tilde{\sm} x$ on $\R^+$.
    Moreover, any operator monotone function on $\R^+$ arises in this form.
\end{prop}
\begin{proof}
    The correspondence between connections on $B(\HH)^+$ and operator monotone functions on $\R^+$
    allows us to consider only operator monotone functions on $\R^+$ constructing from connections on $B(\HH)^+$.
    Proposition \ref{thm: connection gives operator mon} shows that these operator monotone functions take
    the form $x \mapsto 1 \tilde{\sm} x$.
\end{proof}

\begin{prop} \label{prop: extend conn from center}
    Each binary operation $\sigma$ on the center
    \begin{align*}
        \R^+I = \{kI: k \in \R^+\}
    \end{align*}
    of $B(\HH)^+$ can be uniquely extended to a connection on $B(\HH)^+$.
\end{prop}
\begin{proof}
    Let $\tau, \eta$ be two connections on $B(\HH)^+$ which are extensions of $\sigma$. Let $f,g$ be representing functions of $\tau, \eta$, respectively. Then for $x \geqs 0$,
    \begin{align*}
        f(x)I = I \,\tau\, (xI) = I \sm (xI) =I \,\eta\, (xI)  = g(x)I,
    \end{align*}
    i.e. $f=g$. Hence, $\tau = \eta$.
\end{proof}


\begin{thm} \label{thm: con and ind con}
    The map $\sm \mapsto \tilde{\sm}$ from the set of connections on $B(\HH)^+$ to the set of connections on $\R^+$ such that
    \begin{align}
        (x \tilde{\sm} y)I = (xI) \sm (yI), \quad x,y \in \R^+, \label{eq: con and ind con}
    \end{align}
    is an affine order isomorphism. 
\end{thm}
\begin{proof}
    To show that this map is surjective, let $\eta$ be a connection on $\R^+$.
    Define a binary operation $\sigma$ on the center $\R^+I$ of $B(\HH)^+$ by
    \begin{align*}
        (xI) \sm (yI) = (x \,\eta\, y)I, \quad x,y \in \R^+.
    \end{align*}
    Extend it to a connection on $B(\HH)^+$ by Proposition \ref{prop: extend conn from center}.

    Now, suppose $\sigma_i \mapsto \sigma$ for $i=1,2$. Let $f_i$ be the representing function of $\sigma_i$ for $i=1,2$. 
    Then for $x \in \R^+$
    \begin{align*}
        f_1(x)I = I \,\sigma_1\, (xI) = (1 \,\eta\, x)I = I \,\sigma_2\, (xI) = f_2(x)I,
    \end{align*}
    i.e. $f_1=f_2$ and $\sigma_1 = \sigma_2$.
    It is straightforward to check that this map is affine (i.e. it preserves nonnegative linear combinations) and order-preserving.
\end{proof}

\begin{cor} \label{prop: con and ind con have same formula}
    A connection on $B(\HH)^+$ and its induced connection have the same representing function and the same representing measure. 
    More precisely, given an operator monotone function
\begin{align}
		f(x) = \ap + \bt x+ \int_{(0,\infty)} \frac{\ld+1}{2\ld} (\ld \,!\, x) \, d\mu(\ld), \label{eq: formula of f}
\end{align}
one has, for each $A,B \in B(\HH)^+$ and $x,y \in \R^+$,
\begin{align}
	A \sm B &= \ap A+ \bt B+ \int_{(0,\infty)} \frac{ \ld+1}{2\ld} (\ld A \,!\, B) \, d\mu(\ld),
            \label{eq: formula of connec} \\  				
    x \tilde{\sm} y &= \ap x + \bt y+ \int_{(0,\infty)} \frac{ \ld+1}{2\ld} (\ld x \,!\, y) \, d\mu(\ld)
            \label{eq: formula of induc sm}.
\end{align}
\end{cor}
\begin{proof}
    Let $\sigma$ be a connection and $\tilde{\sigma}$ its induced connection.
    Then the correspondences between connections, induced connections, finite Borel measures and operator monotone functions imply that $\sigma$ and $\tilde{\sigma}$
    have the same representing function and the same representing measure.
    Hence, $\sigma$ has the integral representation \eqref{eq: formula of connec}.
    The formula \eqref{eq: formula of induc sm} of $\tilde{\sigma}$ can be computed by using the formula
    \eqref{prop: induced con_prop}. The direct computation shows that the induced connection of the harmonic mean
    on $B(\HH)^+$ is the scalar harmonic mean.
\end{proof}

\begin{cor} \label{cor: mean on a and mean on R}
    A connection is a mean if and only if the induced connection  is a mean on $\R^+$.
\end{cor}
\begin{proof}
    Use Corollary \ref{prop: con and ind con have same formula} and the fact that a connection is a mean if and only if
    its representing function is normalized.
\end{proof}

It is easy to see that the class of means on $B(\HH)^+$ becomes a convex set.

\begin{cor} \label{thm: correspond between 4 obj}
    The map $\sigma \mapsto \tilde{\sigma}$ establishes an affine order isomorphism between
    operator means on $B(\HH)^+$ and scalar means on $\R^+$.
\end{cor}
\begin{proof}
    It is an immediate consequence of Theorem \ref{thm: con and ind con} and
    Corollary \ref{cor: mean on a and mean on R}.
\end{proof}


\begin{remk}
     According to Corollary \ref{thm: correspond between 4 obj}, we can naturally define any named means on $B(\HH)^+$ to be the corresponding means on $\R^+$.
\end{remk}

\begin{ex}
    For each $p \in [-1,1]$ and $\ap \in [0,1]$, the map
\[ x \mapsto [(1-\ap)+\ap x^p]^{1/p} \]
is an operator monotone function on $\R^+$ (when $p=0$, it is understood that we take limit as $p$ approaches $0$).
Hence, it produces a mean on $\R^+$, given by
\begin{align*}
    x \,\#_{p,\ap}\, y = [(1-\ap)x^p + \ap y^p]^{1/p}.
\end{align*}
This is the formula of the \emph{quasi-arithmetic power mean} with parameter $(p,\ap)$.
Now, we define the \emph{quasi-arithmetic power mean} for positive operators $A,B$ on $\HH$
to be the mean on $B(\HH)^+$ corresponds to this scalar mean.
The class of quasi-arithmetic power means contains many kinds of means: The mean $\#_{1,\ap}$ is the $\ap$-weighed arithmetic mean. The case $\#_{0,\ap}$ is the $\ap$-weighed geometric mean. 
The case $\#_{-1,\ap}$ is
the $\ap$-weighed harmonic mean. The mean $\#_{p,1/2}$ is the power mean of order $p$.
\end{ex}

\noindent	\textbf{Acknowledgements}:
The first author is indebted to a financial support from the Chulalongkorn University Graduate School for the 
 Graduate Scholarship to Commemorate the 72nd Anniversary of His Majesty
King Bhumibol Adulyadej and the 90th Anniversary of Chulalongkorn University Fund 
(Ratchadaphiseksomphot Endowment Fund) during his Ph.D. study.

\end{document}